\documentclass[12pt,reqno]{article}

\usepackage[usenames]{color}
\usepackage{amssymb}
\usepackage{graphicx}
\usepackage{amscd}

\usepackage[colorlinks=true,
linkcolor=webgreen,
filecolor=webbrown,
citecolor=webgreen]{hyperref}

\definecolor{webgreen}{rgb}{0,.5,0}
\definecolor{webbrown}{rgb}{.6,0,0}

\usepackage{color}
\usepackage{fullpage}
\usepackage{float}

\usepackage{graphics,amsmath,amssymb}
\usepackage{amsthm}
\usepackage{amsfonts}
\usepackage{latexsym}

\DeclareMathOperator{\Li}{Li}
\DeclareMathOperator{\arctanh}{arctanh}

\begin{document}

\begin{center}
\end{center}

\theoremstyle{plain}
\newtheorem{theorem}{Theorem}
\newtheorem{corollary}[theorem]{Corollary}
\newtheorem{lemma}[theorem]{Lemma}
\newtheorem{proposition}[theorem]{Proposition}

\theoremstyle{remark}
\newtheorem{remark}[theorem]{Remark}

\newcommand{\lrf}[1]{\left\lfloor #1\right\rfloor}
\newcommand{\blue}[1]{{\color{blue}#1}}
\newcommand{\red}[1]{{\color{red}#1}}

\begin{center}
{\large\bf On the Euler transform and the floor function}

\vskip 1cm

{\large
Robert Frontczak \\ Independent Researcher,\\ 72764 Reutlingen, Germany \\
\href{mailto:robert.frontczak@web.de}{\tt robert.frontczak@web.de}
}
\end{center}

\vskip .2 in

\begin{abstract}
Let $(a_n)_{n\geq 0}$ be a sequence of numbers. We derive an expression for the Euler transform of a general binomial sum involving $a_n$
and weighted by the floor function. To demonstrate the usefulness of our approach, several examples are discussed. We rediscover some known identities and prove several new. In particular, we derive some new identities for weighted binomial sums involving harmonic numbers and central binomial coefficients. We also present new closed-forms for weighted series involving the Riemann zeta function.
\end{abstract}

\section{Motivation}

Let $n\geq 0$ be a nonnegative integer. Given a sequence of complex numbers $(a_{n})_{n\geq 0}$, its binomial transform 
$(b_{n})_{n\geq 0}$ is the sequence defined by
$$b_{n} = \sum _{k=0}^{n} \binom{n}{k} a_{k}, \quad\mbox{with inversion}\quad a_{n} = \sum_{k= 0}^{n} \binom{n}{k} (-1)^{n-k} b_{k},$$
\noindent or, in the symmetric version
$$b_{n} = \sum_{k=0}^{n} \binom{n}{k} (-1)^{k} a_{k} \quad\mbox{with inversion}\quad a_{n} = \sum_{k=0}^{n} \binom{n}{k} (-1)^{k}b_{k}.$$
\noindent For more information about binomial transforms consult the book by Boyadzhiev \cite{Boyadzhiev} or the articles by 
Adegoke \cite{Adegoke}, Chen \cite{Chen}, Gould \cite{Gould1}, Gould and Quaintance \cite{Gould2}, Prodinger \cite{Prodinger}, 
Spivey \cite{Spivey} or the Sun brothers \cite{Sun1,Sun2,Sun3}. \\

For the computation of binomial transforms a very efficient tool is Euler’s transformation formula for series (or Euler's transform) 
(\cite{Boyadzhiev2,Boyadzhiev3,Boyadzhiev4,Frontczak,Sondow}): Given a function $f(z)$ holomorphic in a neighborhood of $z=0$
\begin{equation*}
f(z) = \sum_{n= 0}^\infty a_{n} z^{n},
\end{equation*}
then we have the representation
\begin{equation}\label{Euler}
\sum_{n= 0}^\infty z^{n} \left (\sum_{k=0}^{n}\binom{n}{k} a_{k} \right ) = \frac{1}{1-z} f\left (\frac{z}{1-z} \right ).
\end{equation}
With the substitution $t=z/(1-z)$, \eqref{Euler} takes the form
\begin{equation*}
f(t) = \frac{1}{1+t} \sum_{n=0}^\infty \left (\frac{t}{1+t}\right )^n \left (\sum_{k=0}^{n}\binom{n}{k} a_{k} \right ),
\end{equation*}
which is sometimes also useful. For instance, if $f(1)$ exists, then
\begin{equation}\label{Series_rel}
\sum_{n=0}^\infty a_n = \sum_{n=0}^\infty \frac{1}{2^{n+1}} \left (\sum_{k=0}^{n}\binom{n}{k} a_{k} \right ),
\end{equation}
and sometimes the new series on the left converges faster. See examples in Knopp's book~\cite{Knopp}. \\
The representation \eqref{Euler} can be extended to its general form
\begin{equation}\label{Euler_transform}
\sum_{n= 0}^\infty z^{n}  \left ( \sum_{k=0}^{n} \binom{n}{k} \mu^{k} \lambda^{n-k} a_{k} \right ) 
= \frac{1}{1-\lambda z} f\left ( \frac{\mu z}{1-\lambda z} \right ),
\end{equation}
where $\lambda,\mu $ are appropriate parameters. \\

In this article we study sums of the form
$$S_n(\mu,\lambda, a_n) = \sum_{k=0}^{n} \binom{n}{k} \lrf{\frac{k}{2}} \mu^{k} \lambda^{n-k} a_{k},$$
where $\lrf{x}$ is the floor function and $(a_n)_{n\geq 0}$ is an arbitrary sequence of numbers. 
We prove a Euler-type transformation formula for all series associated with $S_n(\mu,\lambda,a_n)$. 
Several examples will show the main results at work. In particular, we will focus on prominent number sequences
like harmonic numbers and central binomial coefficients. We will also present some seemingly new closed-forms for 
series involving the Riemann zeta function, two particular examples being
\begin{equation*}
\sum_{k=2}^{\infty} \lrf{\frac{k}{2}} \frac{\zeta(k)}{k\,2^k} = \frac{\ln(2)+\gamma}{4},
\end{equation*}
and
\begin{equation*}
\sum_{k=2}^{\infty} \lrf{\frac{k}{2}} (-1)^k \frac{\zeta(k)}{2^k} = \frac{\pi^2-4}{16},
\end{equation*}
where $\gamma$ is the Euler-Mascheroni constant.

\section{Main Results}

\begin{theorem}\label{main_thm_Euler}
Let $f(z)$ be a power series given by
\begin{equation*}
f(z) = \sum_{n= 0}^\infty a_{n} z^{n},
\end{equation*}
and holomorphic in a neighborhood of $z=0$. Then we have the representation
\begin{align}\label{main_thm}
&\sum_{n= 0}^\infty z^{n} \left (\sum_{k=0}^{n}\binom{n}{k} \lrf{\frac{k}{2}} \mu^{k} \lambda^{n-k} a_{k} \right ) \nonumber \\
&\qquad = \frac{\mu z}{2(1-\lambda z)^2} f'\left (\frac{\mu z}{1-\lambda z} \right ) 
- \frac{1}{4} \frac{1}{1-\lambda z} \left ( f\left (\frac{\mu z}{1-\lambda z} \right ) - f\left (\frac{-\mu z}{1-\lambda z} \right )\right ).
\end{align}
\end{theorem}
\begin{proof}
We start with the special case $\mu=\lambda=1$ in Theorem \ref{main_thm_Euler}. 
Let $C$ be a positively oriented circle inside the domain of $f$ with radius $R$ and centered at the origin. 
Let $|z|<R$ be such that for $\lambda\in\mathbb{C}$ we have $|z| < |\lambda - z|$. For every $n\ge 0$ we have from Cauchy's integral formula
\begin{equation*}
a_{n} = \frac{1}{2\pi i} \oint_{C} \frac{f(\lambda) d\lambda}{\lambda^{n+1}}  
\end{equation*} 
and therefore,
\begin{equation*}
\sum_{k=0}^{n} \binom{n}{k} \lrf{\frac{k}{2}} a_{k} = \frac{1}{2\pi i} \oint_{C} \left (\sum_{k=0}^{n} \binom{n}{k} \lrf{\frac{k}{2}} \frac{1}{\lambda^k} \right ) \frac{f(\lambda)}{\lambda}\,d\lambda.
\end{equation*}
But it is not difficult to show that
$$\sum_{k=0}^{n} \binom{n}{k} \lrf{\frac{k}{2}} \frac{1}{\lambda^k} = \frac{n}{2} \frac{1}{\lambda}\left (1+\frac{1}{\lambda}\right )^{n-1}
- \frac{1}{4} \left ( \left (1+\frac{1}{\lambda}\right )^{n} - \left (1-\frac{1}{\lambda}\right )^{n} \right ).$$
This gives
\begin{align*}
&\sum_{n=0}^\infty z^{n} \left ( \sum_{k=0}^{n} \binom{n}{k} \lrf{\frac{k}{2}} a_{k} \right ) \nonumber \\
&= \frac{1}{2\pi i} \oint_{C} \left ( \sum_{n=0}^\infty z^{n} \left ( \frac{n}{2} \frac{1}{\lambda}\left (1+\frac{1}{\lambda}\right )^{n-1}
- \frac{1}{4}\left (1+\frac{1}{\lambda}\right )^{n} + \frac{1}{4} \left (1-\frac{1}{\lambda}\right )^{n} \right ) \right ) \frac{f(\lambda)}{\lambda}\,d\lambda \nonumber \\
&= \frac{1}{2\pi i} \oint_{C} \left ( \frac{z}{2\lambda} \sum_{n=0}^\infty n \left (1+\frac{1}{\lambda}\right )^{n-1} z^{n-1} 
- \frac{1}{4} \frac{1}{1-z(1+1/\lambda)} + \frac{1}{4} \frac{1}{1-z(1-1/\lambda)}\right )  \frac{f(\lambda)}{\lambda}\,d\lambda \nonumber \\
&= \frac{1}{2\pi i} \oint_{C} \left ( \frac{z\lambda}{2(1-z)^2\left (\lambda-\frac{z}{1-z}\right )^2} 
- \frac{1}{4} \frac{\lambda}{(1-z)\left (\lambda-\frac{z}{1-z}\right )} + \frac{1}{4} \frac{\lambda}{(1-z)\left (\lambda-\frac{-z}{1-z}\right )}
\right ) \frac{f(\lambda)}{\lambda}\,d\lambda 
\end{align*}
and finally
\begin{align}\label{Euler_sc}
&\sum_{n=0}^\infty z^{n} \left ( \sum_{k=0}^{n} \binom{n}{k} \lrf{\frac{k}{2}} a_{k} \right ) \nonumber \\
&\qquad\qquad = \frac{z}{2(1-z)^2} f'\left (\frac{z}{1-z} \right ) 
- \frac{1}{4} \frac{1}{1-z} f\left (\frac{z}{1-z} \right ) + \frac{1}{4} \frac{1}{1-z} f\left (\frac{-z}{1-z} \right ),
\end{align}
where in the last step we used Cauchy's integral formula for derivatives
$$f^{(n)}(z) = \frac{n!}{2\pi i} \oint_{C} \frac{f(\lambda)}{(\lambda-z)^{n+1}}\, d\lambda.$$
This proves the special case. The general transformation formula is obtained by applying the transformation \eqref{Euler_sc} 
to the function $g(z)=f(\mu/\lambda\,z)$. This produces
\begin{align*}
&\sum_{n=0}^\infty z^{n} \left ( \sum_{k=0}^{n} \binom{n}{k} \lrf{\frac{k}{2}} \left (\frac{\mu}{\lambda} \right )^k a_{k} \right ) \nonumber \\
&\qquad = \frac{z}{2(1-z)^2} f'\left (\frac{\mu}{\lambda} \frac{z}{1-z} \right ) \frac{\mu}{\lambda}
- \frac{1}{4} \frac{1}{1-z} f\left ( \frac{\mu}{\lambda} \frac{z}{1-z} \right ) 
+ \frac{1}{4} \frac{1}{1-z} f\left ( \frac{\mu}{\lambda} \frac{-z}{1-z} \right ),
\end{align*}
and replacing $z$ by $\lambda z$ completes the proof.
\end{proof}

The next corollary contains an Euler-type series relation analogous to \eqref{Series_rel}.
\begin{corollary}\label{Euler_rel}
If $f(1)$, $f(-1)$ and $f'(1)$ exist, then
\begin{equation}\label{Euler2}
\sum_{k=1}^{\infty} \lrf{\frac{k}{2}} a_k = \frac{1}{2} \left ( f'(1) - \sum_{n=0}^\infty a_{2n+1} \right ).
\end{equation}
\end{corollary}
\begin{proof}
We work with the transformation \eqref{Euler_sc}, replace $z$ by $t/(1+t)$, and simplify to get
$$\frac{f(t)-f(-t)}{2} = t\,f'(t) - \frac{2}{1+t} \sum_{n=0}^\infty \left (\frac{t}{1+t}\right )^n \left ( \sum_{k=0}^{n} \binom{n}{k} \lrf{\frac{k}{2}} a_k \right ).$$
Setting $t=1$, while keeping in mind that
$$\frac{f(1)-f(-1)}{2} = \sum_{n=0}^\infty a_{2n+1},$$
we get
\begin{equation*}
\sum_{n=0}^\infty a_{2n+1} = f'(1) - \sum_{n=0}^\infty \frac{1}{2^n} \sum_{k=0}^{n} \binom{n}{k} \lrf{\frac{k}{2}} a_k.
\end{equation*}
But
\begin{align*}
\sum_{n=0}^\infty \frac{1}{2^n} \sum_{k=0}^{n} \binom{n}{k} \lrf{\frac{k}{2}} a_k 
&= \sum_{n=1}^\infty \frac{1}{2^n} \sum_{k=1}^{n} \binom{n}{k} \lrf{\frac{k}{2}} a_k \\
&= \sum_{k=1}^\infty \lrf{\frac{k}{2}} a_k \sum_{n=k}^\infty \frac{1}{2^n} \binom{n}{k} \\
&= \sum_{k=1}^\infty \lrf{\frac{k}{2}} a_k \sum_{n=0}^\infty \frac{1}{2^{n+k}} \binom{n+k}{k} \\
&= 2 \sum_{k=1}^\infty \lrf{\frac{k}{2}} a_k,
\end{align*} 
where in the final step we have used the known binomial expansion
$$\sum_{n=0}^\infty \binom{n+k}{k} z^n = \frac{1}{(1-z)^{k+1}}, \qquad |z|<1.$$
\end{proof}

Theorem \ref{main_thm_Euler} and Corollary \ref{Euler_rel} are useful tools as will be shown  in the following two applications below. 
We begin with a particular application of \eqref{Euler2} and consider for $m\geq 3$ the sequence $a_n=a_n(m)$ given by $a_0=0$ 
and for $n\geq 1, a_n=1/n^m$. This choice gives 
$$f(z) = \Li_m(z) \quad\text{and}\quad f'(z)=\frac{\Li_{m-1}(z)}{z}, \quad |z|\leq 1,$$
where $\Li_m(z)$ denotes the polylogarithm \cite{Lewin}. We have $f(1)=\zeta(m)$ and $f'(1)=\zeta(m-1)$, where 
$\zeta(z)=\sum_{n=1}^\infty 1/n^z,\Re(z)>1,$ is the Riemann zeta function. As
$$\sum_{n=0}^\infty \frac{1}{(2n+1)^m} = (1-2^{-m}) \zeta(m),$$ 
we get the final result valid for all $m\geq 3$
\begin{equation}\label{gap}
\sum_{k=1}^\infty \lrf{\frac{k}{2}} \frac{1}{k^m} = \frac{1}{2} \left ( \zeta(m-1) - (1-2^{-m})\zeta(m) \right ).
\end{equation}
The particular case $m=3$ becomes
$$\sum_{k=1}^\infty \lrf{\frac{k}{2}} \frac{1}{k^3} = \frac{\pi^2}{12} - \frac{7}{16} \zeta(3).$$

\begin{remark}
The series in \eqref{gap} is known. It appeared recently in \cite[Corollary 1]{Somu}, stated there in the equivalent form
\begin{equation}
\sum_{n=1}^\infty \left ( \frac{1}{n^m} + \frac{1}{(n+2)^m} + \frac{1}{(n+4)^m} + \cdots \right ) 
= \frac{1}{2} \left ( \zeta(m-1) - \frac{(2^m-1)}{2^m} \zeta(m) \right ).
\end{equation}
The proof, however, is longer and more intricate, based on properties of the psi and Hurwitz zeta function. The special case for $m=3$
also appeared in \cite[Problem 2.50, part (c)]{Furdui}.
\end{remark}

We conclude this section with the most basic application of Theorem \ref{main_thm_Euler}. Let $a_n=1$ for all $n$. Then,
$$f(z) = \frac{1}{1-z}, \qquad f'(z) = \frac{1}{(1-z)^2},$$
and
$$f\left (\frac{\mu z}{1-\lambda z} \right ) = \frac{1-\lambda z}{1-(\mu+\lambda)z}, \, 
f\left (\frac{-\mu z}{1-\lambda z} \right ) = \frac{1-\lambda z}{1+(\mu-\lambda)z},\,
f'\left (\frac{\mu z}{1-\lambda z} \right ) = \frac{(1-\lambda z)^2}{(1-(\mu+\lambda)z)^2}.$$
Hence,
\begin{align*}
&\sum_{n= 0}^\infty z^{n} \left (\sum_{k=0}^{n}\binom{n}{k} \lrf{\frac{k}{2}} \mu^{k} \lambda^{n-k} \right ) \nonumber \\
&\quad = \frac{1}{4} \left ( \frac{2\mu z}{(1-(\mu+\lambda)z)^2} - \frac{1}{1-(\mu+\lambda)z} + \frac{1}{1+(\mu-\lambda)z}\right ) \nonumber \\
&\quad = \frac{1}{4} \left ( \frac{2\mu z}{(\mu+\lambda)}\left (\frac{1}{1-(\mu+\lambda)z}\right )' - \frac{1}{1-(\mu+\lambda)z} + \frac{1}{1+(\mu-\lambda)z}\right ) \nonumber \\
&\quad = \frac{1}{4} \left (\frac{2\mu}{(\mu+\lambda)}\sum_{n=0}^\infty n(\mu+\lambda)^n\,z^n - \sum_{n=0}^\infty (\mu+\lambda)^n z^n + 
\sum_{n=0}^\infty (\mu-\lambda)^n (-1)^n z^n \right ),
\end{align*}
and by extracting the coefficients we get the identity
\begin{equation}\label{Ex1}
\sum_{k=0}^{n}\binom{n}{k} \lrf{\frac{k}{2}} \mu^{k} \lambda^{n-k} 
= \frac{\mu n}{2}(\lambda + \mu)^{n-1} - \frac{1}{4}\left ((\lambda+\mu)^n - (\lambda-\mu)^n \right ).
\end{equation}

\begin{remark}
Identity \eqref{Ex1} is Theorem 3.1 in \cite{Adegoke2}.
\end{remark}

\section{Further applications}

We continue with a few additional applications of Theorem \ref{main_thm_Euler} and Corollary \ref{Euler_rel},
hereby stating the major findings as propositions. Where necessary we use the conventions $0^0=1$
and $\sum_{k=m}^n s_k=0$ if $n<m$.

\subsection{Applications involving harmonic numbers}

In this section, we work with the sequence $a_n = H_n$, where $(H_n)_{n\geq 0}$ are harmonic numbers defined by $H_0 = 0$ and for $n\geq 1:$
$$H_n = \sum_{k=1}^n \frac{1}{k} = H_{n-1} + \frac{1}{n}.$$
They have the following integral form
$$H_n = \int_0^1 \frac{1-x^n}{1-x} dx,$$
and their ordinary generating function $f(z)$ is
\begin{equation}\label{har_gf}
f(z) = \sum_{n=0}^\infty H_n z^n = - \frac{\ln(1-z)}{1-z}.
\end{equation}

\begin{proposition}
Let $\mu$ and $\lambda$ be two complex parameters. Then we have for all $n\geq 1$
\begin{align}\label{bin_fl_har}
\sum_{k=0}^{n} \binom{n}{k} \lrf{\frac{k}{2}} \mu^{k} \lambda^{n-k} H_{k} &= \frac{n\mu}{2} (\lambda + \mu)^{n-1} H_n 
- \frac{1}{4} \left ( (\lambda+\mu)^n - (\lambda-\mu)^n \right ) H_n \nonumber \\
&\quad - \frac{\mu}{4} \sum_{k=0}^{n-1} \left ( (\lambda+\mu)^k - (\lambda-\mu)^k \right ) \lambda^{n-1-k} H_{n-1-k} \nonumber \\
&\quad - \frac{\mu^2}{2} \sum_{k=0}^{n-2} (k+1) (\lambda+\mu)^k \lambda^{n-2-k} H_{n-2-k}.
\end{align}
\end{proposition}
\begin{proof}
Working with \eqref{har_gf} we see that
$$f'(z) = \frac{1}{1-z} \left (\frac{1}{1-z} - \frac{\ln(1-z)}{1-z}\right )$$
and
$$f'\left (\frac{\mu z}{1-\lambda z} \right ) = \frac{(1-\lambda z)^2}{1-(\mu+\lambda)z} \left ( \frac{1}{1-(\mu+\lambda)z} 
- \frac{\ln(1-(\mu+\lambda)z)}{1-(\mu+\lambda)z} + \frac{\ln(1-\lambda z)}{1-(\mu+\lambda)z} \right ).$$
This yields
\begin{align*}
\frac{\mu z}{2(1-\lambda z)^2} f'\left (\frac{\mu z}{1-\lambda z} \right ) &= \frac{\mu z}{2} \frac{1}{1-(\mu+\lambda)z} 
\Bigg (\frac{1}{1-(\mu+\lambda)z} - \frac{\ln(1-(\mu+\lambda)z)}{1-(\mu+\lambda)z} \\
&\qquad + \frac{\mu z}{1-(\mu+\lambda)z} \frac{\ln(1-\lambda z)}{1-\lambda z} + \frac{\ln(1-\lambda z)}{1-\lambda z} \Bigg ),
\end{align*}
and also
\begin{equation*}
\frac{1}{1-\lambda z} f\left (\frac{\mu z}{1-\lambda z} \right ) = - \frac{\ln(1-(\mu+\lambda)z)}{1-(\mu+\lambda)z}
+ \frac{\mu z}{1-(\mu+\lambda)z} \frac{\ln(1-\lambda z)}{1-\lambda z} + \frac{\ln(1-\lambda z)}{1-\lambda z}
\end{equation*}
and
\begin{equation*}
\frac{1}{1-\lambda z} f\left (\frac{-\mu z}{1-\lambda z} \right ) = - \frac{\ln(1-(\lambda-\mu)z)}{1-(\lambda-\mu)z}
- \frac{\mu z}{1-(\lambda-\mu)z} \frac{\ln(1-\lambda z)}{1-\lambda z} + \frac{\ln(1-\lambda z)}{1-\lambda z}.
\end{equation*}
Switching to power series and making use of Cauchy's product formula we can state these relations as

\begin{align*}
\frac{\mu z}{2(1-\lambda z)^2} f'\left (\frac{\mu z}{1-\lambda z} \right ) &= \frac{\mu z}{2}  
\Bigg ( \sum_{n=0}^\infty (n+1)(\lambda+\mu)^n z^n + \left (\sum_{n=0}^\infty (\lambda+\mu)^n z^n\right )\left (\sum_{n=0}^\infty (\lambda+\mu)^n H_n z^n \right ) \\
&\qquad - \mu z \left (\sum_{n=0}^\infty (n+1)(\lambda+\mu)^n z^n \right )\left (\sum_{n=0}^\infty \lambda^n H_n z^n \right ) \\
&\qquad - \left (\sum_{n=0}^\infty (\lambda+\mu)^n z^n \right )\left (\sum_{n=0}^\infty \lambda^n H_n z^n \right ) \Bigg ) \\
&= \frac{\mu}{2} \Bigg ( \sum_{n=1}^\infty n(\lambda+\mu)^{n-1} z^n + \sum_{n=1}^\infty (\lambda+\mu)^{n-1} \sum_{k=0}^{n-1} H_k z^n \\
&\qquad - \sum_{n=1}^\infty \sum_{k=0}^{n-1} (\lambda+\mu)^{k} \lambda^{n-1-k} H_{n-1-k} z^n \Bigg ) \\
&\qquad - \frac{\mu^2}{2} \sum_{n=2}^\infty \sum_{k=0}^{n-2} (k+1) (\lambda+\mu)^{k} \lambda^{n-2-k} H_{n-2-k} z^n,
\end{align*}
and
\begin{equation*}
\frac{1}{1-\lambda z} f\left (\frac{\mu z}{1-\lambda z} \right ) = \sum_{n=0}^\infty \left ((\lambda+\mu)^n - \lambda^n \right ) H_n z^n
- \mu \sum_{n=1}^\infty \sum_{k=0}^{n-1} (\lambda+\mu)^{k} \lambda^{n-1-k} H_{n-1-k} z^n,
\end{equation*}
\begin{equation*}
\frac{1}{1-\lambda z} f\left (\frac{-\mu z}{1-\lambda z} \right ) = \sum_{n=0}^\infty \left ((\lambda-\mu)^n - \lambda^n \right ) H_n z^n
+ \mu \sum_{n=1}^\infty \sum_{k=0}^{n-1} (\lambda-\mu)^{k} \lambda^{n-1-k} H_{n-1-k} z^n.
\end{equation*}
Extracting the coefficients and making use of the standard sum
$$\sum_{k=0}^n H_k = (n+1)(H_{n+1}-1),$$
results in
\begin{align*}
\sum_{k=0}^{n} \binom{n}{k} \lrf{\frac{k}{2}} \mu^{k} \lambda^{n-k} H_{k} &= \frac{n\mu}{2} (\lambda + \mu)^{n-1} + \frac{\mu}{2} (\lambda + \mu)^{n-1}(n H_{n-1}-n+1) \\ 
&\qquad - \frac{1}{4} \left ((\lambda+\mu)^n - (\lambda-\mu)^n \right ) H_n \nonumber \\
&\qquad - \frac{\mu}{4} \sum_{k=0}^{n-1} \left ( (\lambda+\mu)^k - (\lambda-\mu)^k \right ) \lambda^{n-1-k} H_{n-1-k} \nonumber \\
&\qquad - \frac{\mu^2}{2} \sum_{k=0}^{n-2} (k+1) (\lambda+\mu)^k \lambda^{n-2-k} H_{n-2-k},
\end{align*}
which is the stated identity as $n H_{n-1}+1=n H_n$.
\end{proof}

By slightly modifying the proof of identity \eqref{bin_fl_har} we can obtain the following equivalent expression for the sum 
$S_n(\mu,\lambda,H_n)$, which is a Knuth-Boyadzhiev-type identity.

\begin{proposition}
Let $\mu$ and $\lambda$ be two complex parameters. Then we have for all $n\geq 1$
\begin{align}\label{bin_fl_har2}
& \sum_{k=0}^{n} \binom{n}{k} \lrf{\frac{k}{2}} \mu^{k} \lambda^{n-k} H_{k} \nonumber \\
&\quad = \frac{n\mu}{2} (\lambda + \mu)^{n-1} \left (H_n - \sum_{k=1}^{n-1} \frac{1}{k} \left (\frac{\lambda}{\lambda+\mu} \right )^k \right )
- \frac{\lambda}{2} \left (\lambda^{n-1} - (\lambda+\mu)^{n-1} \right ) \nonumber \\
&\qquad - \frac{1}{4} \left ( (\lambda+\mu)^n - (\lambda-\mu)^n \right ) H_n 
+ \frac{1}{4} \sum_{k=1}^{n} \frac{\lambda^k}{k} \left ( (\lambda+\mu)^{n-k} - (\lambda-\mu)^{n-k} \right ).
\end{align}
\end{proposition}
\begin{proof}
The key ingredient for establishing \eqref{bin_fl_har2} is the following result due to Boyadzhiev \cite[Lemma 3.13]{Boyadzhiev}:
For sufficiently small $|z|$ we have
$$-\frac{\ln(1-\alpha z)}{1-\beta z} = \sum_{n=1}^\infty \sum_{k=1}^n \frac{\alpha^k \beta^{n-k}}{k} \, z^n.$$
Making use of this result with $\alpha=\lambda$ and $\beta=\lambda+\mu$, we can restate the expression for $f'(z)$ as
\begin{align*}
\frac{\mu z}{2(1-\lambda z)^2} f'\left (\frac{\mu z}{1-\lambda z} \right ) &= \frac{\mu}{2}  
\Bigg ( \sum_{n=1}^\infty n (\lambda+\mu)^{n-1} z^n + \sum_{n=1}^\infty (\lambda+\mu)^{n-1} \sum_{k=0}^{n-1} H_k z^n \\
&\qquad - \sum_{n=2}^\infty \sum_{k=1}^{n-1} \frac{\lambda^k (\lambda+\mu)^{n-1-k}}{k} z^n  
- \sum_{n=3}^\infty \sum_{j=1}^{n-2} a_j b_{n-1-j} z^n \Bigg ),
\end{align*}
with
$$a_n = (\lambda+\mu)^n \quad\text{and}\quad b_n = \sum_{k=1}^n \frac{\lambda^k (\lambda+\mu)^{n-k}}{k}.$$
We also have the new expressions
\begin{equation*}
\frac{1}{1-\lambda z} f\left (\frac{\mu z}{1-\lambda z} \right ) = \sum_{n=0}^\infty (\lambda+\mu)^n H_n z^n
- \sum_{n=1}^\infty \sum_{k=1}^{n} \frac{\lambda^k (\lambda+\mu)^{n-k}}{k} z^n,
\end{equation*}
and
\begin{equation*}
\frac{1}{1-\lambda z} f\left (\frac{-\mu z}{1-\lambda z} \right ) = \sum_{n=0}^\infty (\lambda-\mu)^n H_n z^n
- \sum_{n=1}^\infty \sum_{k=1}^{n} \frac{\lambda^k (\lambda-\mu)^{n-k}}{k} z^n.
\end{equation*}
Coefficient extraction gives
\begin{align*}
S_n(\mu,\lambda,H_n) &= \frac{n\mu}{2} (\lambda + \mu)^{n-1} H_n - \frac{1}{4} \left ( (\lambda+\mu)^n - (\lambda-\mu)^n \right ) H_n \\
&\quad + \frac{1}{4} \sum_{k=1}^{n} \frac{\lambda^k}{k} \left ( (\lambda+\mu)^{n-k} - (\lambda-\mu)^{n-k} \right ) - \frac{\mu}{2} \sum_{k=1}^{n-1} \frac{\lambda^k (\lambda+\mu)^{n-1-k}}{k} \\
&\quad - \frac{\mu}{2} \sum_{j=1}^{n-2} \sum_{k=1}^{n-1-j} \frac{\lambda^k (\lambda+\mu)^{n-1-k}}{k}.
\end{align*}
But
$$\sum_{j=1}^{n-2} \sum_{k=1}^{n-1-j} \frac{\lambda^k (\lambda+\mu)^{n-1-k}}{k} 
= \sum_{k=1}^{n-2} (n-1-k) \frac{\lambda^k (\lambda+\mu)^{n-1-k}}{k}$$
and we can write
\begin{align*}
& - \frac{\mu}{2} \left ( \sum_{k=1}^{n-1} \frac{\lambda^k (\lambda+\mu)^{n-1-k}}{k} + \sum_{k=1}^{n-2} (n-1-k) \frac{\lambda^k (\lambda+\mu)^{n-1-k}}{k} \right ) \\
&\qquad = - \frac{\mu}{2} \sum_{k=1}^{n-1} (n-k) \frac{\lambda^k (\lambda+\mu)^{n-1-k}}{k} \\
&\qquad = - \frac{\mu}{2} \left ( n \sum_{k=1}^{n-1} \frac{\lambda^k (\lambda+\mu)^{n-1-k}}{k} - \sum_{k=1}^{n-1} \lambda^k (\lambda+\mu)^{n-1-k} \right ) \\
&\qquad = - \frac{\mu n}{2} \sum_{k=1}^{n-1} \frac{\lambda^k (\lambda+\mu)^{n-1-k}}{k} - \frac{\lambda}{2} \left (\lambda^{n-1} - (\lambda+\mu)^{n-1} \right ),
\end{align*}
where in the last step we have used the geometric series identity
$$\sum_{k=1}^{n} a^k b^{n-k} = \frac{a}{a-b} (a^n - b^n).$$
This completes the proof.
\end{proof}

\begin{corollary}
For $n\geq 1$ we have
\begin{equation}
\sum_{k=0}^{n} \binom{n}{k} \lrf{\frac{k}{2}} H_{k} = (n-1)2^{n-2} H_n - \sum_{k=1}^{n-2} 2^{k-1} (H_{n-1-k}+k H_{n-2-k})
\end{equation}
or equivalently
\begin{equation}
\sum_{k=0}^{n} \binom{n}{k} \lrf{\frac{k}{2}} H_{k} = (n-1) 2^{n-2}\left (H_n - \sum_{k=1}^{n-1} \frac{1}{k 2^k}\right ) + \frac{2^{n-1}-1}{2},
\end{equation}
and for $n\geq 2$
\begin{equation}
\sum_{k=0}^{n} \binom{n}{k} \lrf{\frac{k}{2}} (-1)^k H_{k} = 2^{n-2} \left ( H_n -\sum_{k=1}^{n-1} \frac{1}{k 2^k} \right ) + \frac{1}{2(n-1)}.
\end{equation}
\end{corollary}
\begin{proof}
The first sum is a special case of \eqref{bin_fl_har} for $\mu=\lambda=1$. The equivalent form is obtained from \eqref{bin_fl_har2} for the same values of $\mu$ and $\lambda$. Setting $\mu=-1$ and $\lambda=1$ in \eqref{bin_fl_har} produces
$$\sum_{k=0}^{n} \binom{n}{k} \lrf{\frac{k}{2}} (-1)^k H_{k} = \frac{1}{4} \left ( 2^n H_n + H_{n-1} - 2 H_{n-2} \right )
- \sum_{k=1}^{n-1} 2^{k-2} H_{n-1-k},$$
and the result follows immediately by applying convolution
$$\sum_{k=0}^n 2^{k-2} H_{n-k} = 2^{n-1} \sum_{k=1}^n \frac{1}{k 2^k} - \frac{1}{4} H_n.$$ It follows also directly from identity 
\eqref{bin_fl_har2}.
\end{proof}

\begin{remark}
The above results can be compared with the classical sums \cite[Eqs. (3.24) and~(3.31)]{Boyadzhiev}
$$\sum_{k=0}^{n} \binom{n}{k} H_{k} = 2^{n} \left ( H_n -\sum_{k=1}^{n} \frac{1}{k 2^k} \right ) \quad\text{and}\quad 
\sum_{k=0}^{n} \binom{n}{k} (-1)^k H_{k} = - \frac{1}{n}.$$
\end{remark}

\begin{corollary}
For $n\geq 1$ we have
\begin{align}
\sum_{k=0}^{n} \binom{n}{k} \lrf{\frac{k}{2}} (-1)^{n-k} 2^k H_{k} &= 2n O_{\lrf{n/2}} - \frac{1}{2} O_{\lrf{(n+1)/2}} + \frac{1+(-1)^{n-1}}{2} \nonumber \\
&\qquad + \frac{(-1)^n}{4} 3^{n} \left ( H_n -\sum_{k=1}^{n} \frac{1}{k 3^k} \right ),
\end{align}
where $O_n$ are odd harmonic numbers, $O_n=\sum_{j=1}^n 1/(2j-1).$
\end{corollary}
\begin{proof}
Setting $\mu=2$ and $\lambda=-1$ in \eqref{bin_fl_har2} results after some steps of simplifications in
\begin{align*}
\sum_{k=0}^{n} \binom{n}{k} \lrf{\frac{k}{2}} (-1)^{n-k} 2^k H_{k} &= n (H_{n-1}+\overline{H}_{n-1}) - \frac{1}{4}(H_{n}+\overline{H}_{n}) 
+ \frac{1+(-1)^{n-1}}{2} \\
&\qquad + \frac{(-1)^n}{4} 3^{n} \left ( H_n -\sum_{k=1}^{n} \frac{1}{k 3^k} \right ),
\end{align*}
where $\overline{H}_n$ are alternating harmonic numbers given by
$$\overline{H}_n = \sum_{j=1}^n \frac{(-1)^{j-1}}{j}, \qquad \overline{H}_0 = 0.$$
This is the stated identity as we have the relation $H_{n}+\overline{H}_{n}=2O_{\lrf{(n+1)/2}}$.
\end{proof}

\subsection{Applications involving central binomial coefficients}

If we consider the sequence $a_n = \binom{2n}{n}$, then we have 
\begin{equation}\label{cbc_gf}
f(z) = \sum_{n=0}^\infty \binom{2n}{n} z^n = \frac{1}{\sqrt{1-4z}}, \qquad |z|<1/4.
\end{equation}

\begin{proposition}
Let $\mu$ and $\lambda$ be two complex parameters. Then we have for all $n\geq 1$
\begin{align}\label{bin_fl_cbc}
&\sum_{k=0}^{n} \binom{n}{k} \binom{2k}{k} \lrf{\frac{k}{2}} \mu^{k} \lambda^{n-k} \nonumber \\
&\quad = \mu 2^{-2n+1} \sum_{k=0}^{n} \binom{2k}{k} \binom{2(n-k)}{n-k} (n-k) \lambda^{k} (\lambda+4\mu)^{n-1-k} \nonumber \\
&\qquad - 2^{-(2n+2)} \sum_{k=0}^{n} \binom{2k}{k} \binom{2(n-k)}{n-k} \lambda^{k} \left ( (\lambda+4\mu)^{n-k} - (\lambda-4\mu)^{n-k} \right ).
\end{align}
\end{proposition}
\begin{proof}
From \eqref{cbc_gf} we get
$$f'(z) = \frac{2}{(1-4z)^{3/2}}.$$
Thus,
\begin{equation*}
\frac{\mu z}{2(1-\lambda z)^2} f'\left (\frac{\mu z}{1-\lambda z} \right ) = \frac{\mu z}{\sqrt{1-\lambda z}(1-(\lambda+4\mu)z)^{3/2}},
\end{equation*}
and also
\begin{equation*}
\frac{1}{1-\lambda z} f\left (\frac{\mu z}{1-\lambda z} \right ) = \frac{1}{\sqrt{1-\lambda z}\sqrt{1-(\lambda+4\mu)z}},
\end{equation*}
as well as
\begin{equation*}
\frac{1}{1-\lambda z} f\left (\frac{-\mu z}{1-\lambda z} \right ) = \frac{1}{\sqrt{1-\lambda z}\sqrt{1-(\lambda-4\mu)z}}.
\end{equation*}
Switching to power series we obtain
\begin{align*}
&\sum_{n= 0}^\infty z^{n} \left (\sum_{k=0}^{n}\binom{n}{k}\binom{2k}{k} \lrf{\frac{k}{2}} \mu^{k} \lambda^{n-k} \right ) \nonumber \\
&\quad = \frac{\mu}{2} \left ( \sum_{n=0}^\infty \binom{2n}{n} \left (\frac{\lambda}{4}\right )^n z^n \right ) \left (\sum_{n=0}^\infty 
\binom{2n}{n} n \left (\frac{\lambda+4\mu}{4}\right )^{n-1} z^n \right ) \nonumber \\
&\qquad - \frac{1}{4} \left ( \sum_{n=0}^\infty \binom{2n}{n} \left (\frac{\lambda}{4}\right )^n z^n \right ) \left (\sum_{n=0}^\infty 
\binom{2n}{n} \left (\frac{\lambda+4\mu}{4}\right )^{n} z^n \right ) \nonumber \\
&\qquad + \frac{1}{4} \left ( \sum_{n=0}^\infty \binom{2n}{n} \left (\frac{\lambda}{4}\right )^n z^n \right ) \left (\sum_{n=0}^\infty 
\binom{2n}{n} \left (\frac{\lambda-4\mu}{4}\right )^{n} z^n \right ).
\end{align*}
The statement follows from Cauchy's product formula, after coefficient extraction and obvious rearrangement.
\end{proof}

\begin{corollary}
We have
\begin{equation}
\sum_{k=0}^{n} \binom{n}{k} \binom{2k}{k} \lrf{\frac{k}{2}} 2^{-2k} = 2^{-2(n+1)} \sum_{k=2}^n \binom{2k}{k} \binom{2(n-k)}{n-k} 2^{k}(k-1)  
\end{equation}
and
\begin{align}
&\sum_{k=0}^{n} \binom{n}{k} \binom{2k}{k} \lrf{\frac{k}{2}} (-1)^{k} 2^{-2k} \nonumber \\
&\qquad = 2^{-(n+2)} \sum_{k=0}^n \binom{2k}{k} \binom{2(n-k)}{n-k} 2^{-k} - 2^{-(2n+2)} \left (\binom{2n}{n} + 4\binom{2(n-1)}{n-1}\right ).
\end{align}
\end{corollary}
\begin{proof}
Set $\lambda=4\mu$ and $\lambda=-4\mu$, in turn, in \eqref{bin_fl_cbc} and simplify.
\end{proof}

Corollary \ref{Euler_rel} can be also applied to central binomial coefficients. To ensure the existence of the necessary quantities, 
we consider the scaled sequence $a_n=\binom{2n}{n} 2^{-2n} p^{-n}$, where $p$ is a real (complex) parameter with $|p|\geq 2$.

\begin{proposition}
We have
\begin{equation}\label{p_series}
\sum_{k=0}^{\infty} \binom{2k}{k} \lrf{\frac{k}{2}} 2^{-2k} p^{-k} 
= \frac{1}{4} \left ( \frac{1}{p}\left (\frac{p}{p-1}\right )^{3/2} + \frac{(\sqrt{p-1}-\sqrt{p+1})\sqrt{p}}{\sqrt{(p-1)(p+1)}} \right ).
\end{equation}
In particular, we have the evaluations
\begin{equation}
\sum_{k=0}^{\infty} \binom{2k}{k} \lrf{\frac{k}{2}} 2^{-3k} = \frac{1}{2\sqrt{2}\sqrt{3}},
\end{equation}
and
\begin{equation}
\sum_{k=0}^{\infty} \binom{2k}{k} \lrf{\frac{k}{2}} (-1)^k 2^{-3k} = \frac{1}{2\sqrt{2}} - \frac{\sqrt{2}}{3\sqrt{3}}.
\end{equation}
\end{proposition}
\begin{proof}
For the scaled sequence $a_n=\binom{2n}{n} 2^{-2n} p^{-n}$ we have from \eqref{cbc_gf}
$$f(z) = \sqrt{\frac{p}{p-z}},\qquad |z|<|p|.$$
Hence,
$$f(1) = \sqrt{\frac{p}{p-1}} \qquad\text{and}\qquad f'(1) = \frac{1}{2p}\left (\frac{p}{p-1}\right )^{3/2}.$$
Finally, 
\begin{align*}
\sum_{n=0}^\infty a_{2n+1} &= \sum_{n=0}^\infty \binom{2(2n+1)}{2n+1} 2^{-2(2n+1)} p^{-(2n+1)} \\
&= - \frac{\sqrt{\frac{p-1}{p}}-\sqrt{\frac{p+1}{p}}}{2 \sqrt{\frac{p-1}{p}}\sqrt{\frac{p+1}{p}}} 
= \frac{(\sqrt{p+1}-\sqrt{p-1})\sqrt{p}}{2\sqrt{(p-1)(p+1)}},
\end{align*}
and the statement follows from Corollary \ref{Euler_rel}. The particular sums are special cases of \eqref{p_series} for $p=\pm 2$.
\end{proof}

\subsection{Series involving zeta function}

In this part, we heavily relay on the book by Srivastava and Choi \cite{Srivastava} for all relevant results involving 
the Riemann zeta function $\zeta(z) = \sum_{k=1}^\infty \frac{1}{k^z}, \Re(z)>1.$ In addition, we will apply properties 
of the digamma function $\psi(z), z\in\mathbb{C}$, which is the first logarithmic derivative of the Gamma function, i.e., 
$\psi(z) = (\ln \Gamma(z))' = \Gamma'(z)/\Gamma(z),$ where $\Gamma(z)$ is the complex gamma function defined for $\Re(z)>0$ 
by the integral
$$\Gamma (z) = \int_0^\infty e^{- t} t^{z - 1}dt,$$
and is extended to the rest of the complex plane, excluding the non-positive integers, by analytic continuation. The constant
$$\gamma = -\psi(1) = \lim_{n\rightarrow \infty} \left (H_n - \ln(n) \right ) \approx 0.57721566490\ldots,$$
is the familiar Euler-Mascheroni constant.

\begin{proposition}
We have
\begin{equation}\label{zeta1}
\sum_{k=2}^{\infty} \lrf{\frac{k}{2}} \frac{\zeta(k)}{2^k} = \frac{\pi^2+4}{16},
\end{equation}
and
\begin{equation}\label{zeta2}
\sum_{k=2}^{\infty} \lrf{\frac{k}{2}} \frac{\zeta(k)}{2^{2k}} 
= \frac{1}{4}\left (1 + \frac{\pi^2}{8} - \frac{\pi}{4} - G \right ),
\end{equation}
where $G$ is the Catalan constant,
$$G = \sum_{k=0}^\infty \frac{(-1)^k}{(2k+1)^2}.$$
\end{proposition}
\begin{proof}
We work with \cite[page 160]{Srivastava}
\begin{equation}\label{zeta_gf1}
\sum_{n=2}^\infty \zeta(n)\,z^n = -z (\psi(1-z)+\gamma), \qquad |z|<1.
\end{equation}
We first prove \eqref{zeta1} applying Corollary \ref{Euler_rel}. We set $a_0=a_1=0$ and for $n\geq 2, a_n=\zeta(n)/2^n$.
Then, from \eqref{zeta_gf1}, we see that
$$f(z) = - \frac{z}{2} \left (\psi \left (1-\frac{z}{2}\right )+\gamma \right ), \qquad |z|<2,$$
and
\begin{align*}
f'(1) &= - \frac{1}{2} \left (\psi \left (\frac{1}{2}\right ) + \gamma \right ) + \frac{1}{4} \psi'\left (\frac{1}{2}\right ) \\ 
&= - \frac{1}{2} \left (-\gamma - 2\ln(2) + \gamma \right ) + \frac{\pi^2}{8} \\
&= \ln(2) + \frac{\pi^2}{8}.
\end{align*}
Also,
$$\sum_{n=0}^\infty a_{2n+1} = \sum_{n=1}^\infty \frac{\zeta(2n+1)}{2^{2n+1}} = \frac{1}{2} \sum_{n=1}^\infty \frac{\zeta(2n+1)}{2^{2n}} 
= \ln(2) - \frac{1}{2},$$
where we have used \cite[Equation (46), p. 163]{Srivastava}. This proves \eqref{zeta1}. To establish \eqref{zeta2}, we set $a_0=a_1=0$ 
and for $n\geq 2, a_n=\zeta(n)/2^{2n}$.
Then, again from \eqref{zeta_gf1}, we see that
$$f(z) = - \frac{z}{4} \left (\psi \left (1-\frac{z}{4}\right )+\gamma \right ), \qquad |z|<4.$$
For the derivative $f'(1)$ we get
\begin{align*}
f'(1) &= - \frac{1}{4} \left (\psi \left (\frac{3}{4}\right ) + \gamma \right ) + \frac{1}{16} \psi'\left (\frac{3}{4}\right ) \\ 
&= - \frac{1}{4} \left (- \gamma + \frac{\pi}{2} - 3\ln(2) + \gamma \right ) + \frac{1}{16}\left ( \pi^2 - 8 G \right ) \\
&= - \frac{\pi}{8} + \frac{3}{4}\ln(2) + \frac{\pi^2}{16} - \frac{G}{2},
\end{align*}
where in the second step we have used an evaluation of $\psi'(3/4)$ due to K\"olbig \cite{Kolbig}. Next, we can use
$$\sum_{n=2}^\infty \zeta(2n+1) z^{2n} = - \frac{1}{2} \left ( \psi(1+z)+\psi(1-z) \right ) - \gamma, \quad |z|<1,$$
in conjunction with the functional equation of the digamma function
\begin{equation}\label{fkt_eq}
\psi(1+z) = \psi(1-z) + \frac{1}{z} - \pi \cot(\pi z),
\end{equation}
to evaluate
$$\sum_{n=0}^\infty a_{2n+1} = \frac{1}{4} \sum_{n=1}^\infty \frac{\zeta(2n+1)}{2^{4n}} = - \frac{1}{8} \left (\psi \left (\frac{5}{4}\right ) 
+ \psi \left (\frac{3}{4}\right ) \right ) - \frac{\gamma}{4} = - \frac{1}{2} + \frac{3}{4} \ln(2).$$
This completes the proof of \eqref{zeta2}.
\end{proof}

\begin{proposition}
We have
\begin{equation}\label{zeta3}
\sum_{k=2}^{\infty} \lrf{\frac{k}{2}} \frac{\zeta(k)}{k\,2^k} = \frac{\ln(2)+\gamma}{4},
\end{equation}
and
\begin{equation}\label{zeta4}
\sum_{k=2}^{\infty} \lrf{\frac{k}{2}} \frac{\zeta(k)}{k\,2^{2k}} 
= \frac{1}{4}\left ( 2\ln \Gamma \left(\frac{1}{4}\right ) + \frac{\gamma}{2} - \frac{\pi}{4} - \ln(2\pi) \right ).
\end{equation}
\end{proposition}
\begin{proof}
Here, we work with \cite[page 160]{Srivastava}
\begin{equation}\label{zeta_gf2}
\sum_{n=2}^\infty \frac{\zeta(n)}{n}\,z^n = \ln \Gamma(1-z) - \gamma z, \qquad |z|<1.
\end{equation}
Setting $a_0=a_1=0$ and for $n\geq 2, a_n=\zeta(n)/(n\,2^n)$, \eqref{zeta_gf2} shows that
$$f(z) = \ln \Gamma\left (1-\frac{z}{2}\right ) - \frac{\gamma z}{2}, \qquad |z|<2.$$
From here, we calculate
$$f'(1) = - \frac{1}{2} \psi \left (\frac{1}{2}\right ) - \frac{\gamma}{2} = \ln(2)$$
and using \cite[Equation (27), p. 162]{Srivastava}
$$\sum_{n=0}^\infty a_{2n+1} = \frac{1}{2} \sum_{n=1}^\infty \frac{\zeta(2n+1)}{(2n+1)2^{2n}} = \frac{1}{2}\left (\ln(2)-\gamma \right ).$$
This yields \eqref{zeta3}. To prove \eqref{zeta4}, set $a_0=a_1=0$ and for $n\geq 2, a_n=\zeta(n)/(n\,2^{2n})$. Then
$$f'(1) = - \frac{1}{4} \psi \left (\frac{3}{4}\right ) - \frac{\gamma}{4} = \frac{3}{4} \ln(2) - \frac{\pi}{8}.$$
To finish the proof we apply \cite[Equation (12), p. 160]{Srivastava}
$$\sum_{n=1}^\infty \frac{\zeta(2n+1)}{2n+1} z^{2n+1} = \frac{1}{2} \ln \frac{\Gamma(1-z)}{\Gamma(1+z)} - \gamma z, \qquad |z|<1,$$
and calculate
$$\sum_{n=0}^\infty a_{2n+1} = \frac{1}{4} \sum_{n=1}^\infty \frac{\zeta(2n+1)}{(2n+1) 2^{4n}} = \frac{1}{2} \ln \frac{\Gamma(3/4)}{\Gamma(5/4)} - \frac{\gamma}{4} = \frac{1}{2} \ln \frac{4 \sqrt{2}\pi}{\Gamma^2(1/4)} - \frac{\gamma}{4}.$$
The identity \eqref{zeta4} follows from Corollary \ref{Euler_rel} after some simplifications.
\end{proof}

\begin{proposition}
We have
\begin{equation}\label{zeta5}
\sum_{k=2}^{\infty} \lrf{\frac{k}{2}} (\zeta(k)-1) = \frac{9 + 2\pi^2}{24},
\end{equation}
\begin{equation}\label{zeta6}
\sum_{k=2}^{\infty} \lrf{\frac{k}{2}} \frac{\zeta(k)-1}{k} = \frac{2\gamma + \ln(2)}{4},
\end{equation}
\begin{equation}\label{zeta7}
\sum_{k=2}^{\infty} \lrf{\frac{k}{2}} \frac{\zeta(k)-1}{k 2^k} = \frac{\gamma + \ln(6) - 2}{4},
\end{equation}
and
\begin{equation}\label{zeta8}
\sum_{k=2}^{\infty} \lrf{\frac{k}{2}} \frac{\zeta(k)-1}{k 2^{2k}} 
= \frac{1}{2}\left ( \ln \Gamma \left(\frac{1}{4}\right ) + \frac{\gamma}{4} + \arctanh \left(\frac{1}{4}\right ) - \frac{\pi}{8} - \frac{1}{3}- \frac{1}{2} \ln(2\pi) \right ).
\end{equation}
\end{proposition}
\begin{proof}
Work with \cite[Equation (135),page 173]{Srivastava}
$$\sum_{n=2}^\infty (\zeta(n)-1)\frac{z^n}{n} = \ln \Gamma(2-z) + z(1-\gamma), \qquad |z|<2,$$
in conjunction with \cite[Equation (194), p. 178]{Srivastava}
$$\sum_{n=1}^\infty (\zeta(2n+1)-1) = \frac{1}{4}.$$
Note that to finish the proofs of \eqref{zeta7} and \eqref{zeta8} we need the values $\psi(3/2)$ and $\psi(7/4)$. 
The calculations are done again using the functional equation of the digamma function \eqref{fkt_eq}.
\end{proof}

Many more series of this nature can be evaluated. We conclude with an alternating example (which is stated without proof).

\begin{proposition}
We have
\begin{equation}\label{zeta9}
\sum_{k=2}^{\infty} \lrf{\frac{k}{2}} (-1)^k \frac{\zeta(k)}{2^k} = \frac{\pi^2-4}{16}
\end{equation}
and hence
\begin{equation}\label{zeta10}
\sum_{k=1}^{\infty} \lrf{\frac{2k+1}{2}} \frac{\zeta(2k+1)}{2^{2k}} = \frac{1}{2}.
\end{equation}
\end{proposition}

\section{Concluding Comments}

This article is concerned with binomial transforms for sequences $(a_n)_{n\geq 0}$ weighted by the floor function.
The first main result presented in Theorem \ref{main_thm_Euler} is an Euler-type series representation that allows us to prove new combinatorial results for this transform. The second main result presented in Corollary \ref{Euler_rel} is another Euler-type formula 
that can be applied to infinite series. To keep the article readable, we have restricted the list of examples to three prominent classes. Obviously, the field of applications is much broader, and sums and series involving other famous sequences like Catalan numbers, Fibonacci numbers, Bernoulli numbers and others can be studied in the future.


\bigskip
\hrule
\bigskip

\noindent 2020 {\it Mathematics Subject Classification}:
Primary 05A15, 05A19, 11B83.

\noindent \emph{Keywords: }
Euler transform, floor function, binomial sum, Harmonic number, central binomial coefficient, Riemann zeta function.

\bigskip
\hrule
\bigskip







\begin{thebibliography}{13}

\bibitem{Adegoke}
K. Adegoke, Binomial transforms and the binomial convolution of sequences, Preprint, (2025), 84 pages. https://arxiv.org/pdf/2507.04179

\bibitem{Adegoke2}
K. Adegoke, R. Frontczak and T. Goy, Series and sums involving the floor function, {\it Montes Taurus J. Pure Appl. Math.} {\bf 8} (2026), 
1--26.

\bibitem{Boyadzhiev}
K. N. Boyadzhiev, {\it Notes on the Binomial Transform}, World Scientific, 2018.

\bibitem{Boyadzhiev2} 
K. N. Boyadzhiev, Harmonic number identities via Euler’s transform, {\it J. Integer Seq.} {\bf 12} (2009), Article 09.6.1.

\bibitem{Boyadzhiev3}
K. N. Boyadzhiev, Series transformation formulas of Euler type, Hadamard product of series, and harmonic number identities,
{\it Indian J. Pure Appl. Math.} {\bf 42} (2011), 371--386.

\bibitem{Boyadzhiev4}
K. N. Boyadzhiev and R. Frontczak, Hadamard product of series with special numbers, {\it Funct. Approx. Comment. Math.} {\bf 68} (2023), 
231--247.

\bibitem{Chen}
K. W. Chen, Identities from the binomial transform, {\it J. Number Theory} {\bf 124} (2007), 142--150.

\bibitem{Frontczak}
R. Frontczak, Harmonic sums via Euler’s transform: Complementing the approach of Boyadzhiev, {\it J. Integer Seq.} {\bf 23} (2020), 
Article 20.3.2.

\bibitem{Gould1}
H. W. Gould, Series transformations for finding recurrences for sequences, {\it Fibonacci Quart.} {\bf 28} (1990), 166--171.

\bibitem{Gould2}
H. W. Gould and J. Quaintance, Bernoulli numbers and a new binomial transform identity, {\it J. Integer Seq.} {\bf 17} (2014), Article 14.2.2.

\bibitem{Knopp}
K. Knopp, \textit{Theory and Application of Infinite Series}, Dover, New York, 1990.

\bibitem{Kolbig}
K. S. K\"{o}lbig, The polygamma function $\psi^{(k)}(x)$ for $x=1/4$ and $x=3/4$, {\it J. Comput. Appl. Math.} {\bf 75} (1996) 43--46.

\bibitem{Lewin}
L. Lewin, \textit{Polylogarithms and Associated Functions}, North-Holland, Amsterdam, 1981,

\bibitem{Prodinger}
H. Prodinger, Some information about the binomial transform, {\it Fibonacci Quart.} {\bf 32} (1994), 412--415.

\bibitem{Furdui}
A. S\^{i}nt\u{a}m\u{a}rian and O. Furdui, \textit{Sharpening Mathematical Analysis Skills}, Springer, Berlin/Heidelberg, 2021.

\bibitem{Somu}
S. T. Somu, J. Haw, V. Nguyen and D. V. K. Tran, On some series with gaps, {\it J. Math. Anal. Appl.} {\bf 528} (2023), 127479.

\bibitem{Sondow}
J. Sondow, Analytic continuation of Riemann's zeta function and values at negative integers via Euler's transformation of series, 
{\it Proc. Amer. Math. Soc.} {\bf 120} (1994), 421--424.

\bibitem{Spivey}
M. Z. Spivey, Combinatorial sums and finite differences, {\it Discrete Math.} {\bf 307} (2007), 3130--3146.

\bibitem{Srivastava}
H. M. Srivastava and J. Choi, Series Associated with the Zeta and Related Functions, 
Dordrecht, Boston and London: Kluwer Academic Publishers, 2001.

\bibitem{Sun1}
Z.-H. Sun, Self-inverse sequences under binomial transformation, {\it Fibonacci Quart.} {\bf 39} (2001), 324--333.

\bibitem{Sun2}
Z.-W. Sun, Combinatorial identities in dual sequences, {\it European J. Comb.} {\bf 24} (2003), 709--718.

\bibitem{Sun3}
Z.-H. Sun, Some further properties of even and odd sequences, {\it Int. J. Number Theory} {\bf 13} (2017), 1419--1442.

\end{thebibliography}
\end{document}